\documentclass[12pt,twoside]{amsart}
\usepackage{amssymb,amsmath,amsthm, amscd, enumerate, mathrsfs}
\usepackage{graphicx, hhline}
\usepackage[all]{xy}

\title[Log surfaces]{On minimal model 
theory for algebraic log surfaces}
\author{Osamu Fujino}
\date{2020/12/2, version 0.10}
\address{Department of Mathematics, Graduate School of Science, 
Osaka University, Toyonaka, Osaka 560-0043, Japan}
\email{fujino@math.sci.osaka-u.ac.jp}

\DeclareMathOperator{\Supp}{Supp}
\DeclareMathOperator{\Pic}{Pic}
\DeclareMathOperator{\NE}{\overline{NE}}

\newtheorem{thm}{Theorem}[section]
\newtheorem{lem}[thm]{Lemma}
\newtheorem{cor}[thm]{Corollary}

\theoremstyle{definition}
\newtheorem{ex}[thm]{Example}
\newtheorem{defn}[thm]{Definition}
\newtheorem*{ack}{Acknowledgments}  
\newtheorem{case}{Case}

\makeatletter
    
    \@addtoreset{equation}{section}
\makeatother

\begin{document}

\begin{abstract}
We introduce the notion of generalized MR log canonical 
surfaces and establish the minimal model theory for 
generalized MR log canonical surfaces  
in full generality. 
\end{abstract}

\maketitle 
\tableofcontents

\section{Introduction}\label{a-sec1}

In this short paper, we give some remarks 
on the minimal model theory for log surfaces defined over 
an algebraically closed field $k$ of any characteristic. 
We recall the definition of {\em{log surfaces}} following \cite{fujino-surfaces}. 

\begin{defn}[Log surfaces]\label{a-def1.1} 
Let $X$ be a normal algebraic surface and let $\Delta$ be an 
effective $\mathbb R$-divisor whose coefficients 
are less than or equal to one such that 
$K_X+\Delta$ is $\mathbb R$-Cartier. 
Then $(X, \Delta)$ is called a {\em{log surface}}. 
If $(X, \Delta)$ is a log surface and $X$ is $\mathbb Q$-factorial, 
then we simply say that $(X, \Delta)$ is 
a {\em{$\mathbb Q$-factorial log surface}}. 
\end{defn}

In \cite{fujino-surfaces} and \cite{tanaka} (see also \cite{fujino-tanaka}), 
Osamu Fujino and Hiromu Tanaka established the 
minimal model theory for $\mathbb Q$-factorial log surfaces 
in full generality. We note that a $\mathbb Q$-factorial 
log surface is not necessarily log canonical. 

\begin{defn}[Log canonical surfaces]\label{a-def1.2}
Let $(X, \Delta)$ be a log surface and let $f\colon Y\to X$ 
be a proper birational morphism 
from a normal surface $Y$. 
In this situation, we can write $$K_Y+\Delta_Y=f^*(K_X+\Delta)$$ with 
$f_*\Delta_Y=\Delta$. 
If the coefficients of $\Delta_Y$ are less than or equal to 
one for every $f\colon Y\to X$, then $(X, \Delta)$ 
is called a {\em{log canonical surface}}. 
\end{defn}

We note that a log canonical surface is not always 
$\mathbb Q$-factorial. Therefore, we can not directly apply 
the minimal model theory for $\mathbb Q$-factorial log 
surfaces to log canonical surfaces. Fortunately, however, 
we know that the minimal model theory works in full generality 
for log canonical surfaces by \cite{fujino-surfaces} and \cite{tanaka}. 

Let $\varphi\colon  X\to W$ be a $(K_X+\Delta)$-negative 
extremal birational contraction morphism 
from a log canonical surface $(X, \Delta)$. 
Then the exceptional locus of $\varphi$ is disjoint from 
non-rational singular points of $X$ (see \cite[Section 4]{fujino-rational}). 
This is a key point of the 
minimal model theory for (not necessarily $\mathbb Q$-factorial) 
log canonical surfaces. 

\medskip 

In \cite{alexeev}, Valery Alexeev introduced the notion of 
{\em{MR log canonical surfaces}}, which is a 
slight generalization of that of log canonical surfaces. 
Then, in \cite[Section 10]{alexeev}, he sketched 
the minimal model theory for MR log canonical surfaces. 

\begin{defn}[MR log canonical surfaces]\label{a-def1.3}
Let $(X, \Delta)$ be a log surface and let $f\colon Y\to X$ be 
the unique minimal resolution of singularities of 
$X$. Then we can write 
$$K_Y+\Delta_Y=f^*(K_X+\Delta)$$ with 
$f_*\Delta_Y=\Delta$. If 
the coefficients of $\Delta_Y$ are less than or equal to 
one, 
then we say that 
$(X, \Delta)$ is an {\em{MR log canonical}} 
({\em{MRLC}}, for short) {\em{surface}}. 
\end{defn}

We note that $\Delta_Y$ in Definition \ref{a-def1.3} is always 
effective by the negativity lemma since $f\colon Y\to X$ is the minimal 
resolution of singularities of $X$. 
By definition, a log canonical surface 
is obviously an MR log canonical surface. 
However, an easy example (see Example \ref{a-ex3.3}) shows that 
there exists an MR log canonical surface which is not 
log canonical. Therefore, the minimal model theory 
established in \cite{fujino-surfaces} and \cite{tanaka} 
(see also \cite{fujino-tanaka}) does not cover 
the minimal model theory for MR log canonical surfaces 
sketched in \cite[Section 10]{alexeev}. In this paper, we introduce 
the following new notion of {\em{generalized MR log canonical surfaces}}. 

\begin{defn}[Generalized MR log canonical surfaces]\label{a-def1.4}
Let $(X, \Delta)$ be a log surface. 
If there exists a proper birational morphism 
$f\colon Y\to X$ from a $\mathbb Q$-factorial 
normal surface $Y$ such that $\Delta_Y$ 
is an effective $\mathbb R$-divisor on $Y$ and 
that 
the coefficients of $\Delta_Y$ are less than or equal to one, 
where $$K_Y+\Delta_Y=f^*(K_X+\Delta)$$ with $f_*\Delta_Y=\Delta$, 
then we say that 
$(X, \Delta)$ is a {\em{generalized MR log canonical}} 
({\em{GMRLC}}, for short) {\em{surface}}. 
\end{defn}

By definition, $\mathbb Q$-factorial log 
surfaces, log canonical surfaces, and MR log 
canonical surfaces are all generalized MR log 
canonical surfaces. 
Then we establish the minimal model theory for 
generalized MR log canonical surfaces as follows. 

\begin{thm}[Minimal model program for generalized 
MR log canonical surfaces]\label{a-thm1.5}
Let $(X, \Delta)$ be a generalized MR log canonical 
surface and let $\pi\colon X\to S$ be 
a projective morphism onto a variety $S$. 
Then we can run the minimal model program over $S$ with 
respect to $K_X+\Delta$ by the cone and contraction theorem. 
Hence there exists a sequence of at most $\rho(X/S)-1$ 
birational contractions starting from $(X, \Delta)$
$$
(X, \Delta)=:(X_0, \Delta_0)\overset{\varphi_0}{\longrightarrow} 
(X_1, \Delta_1) \overset{\varphi_1}{\longrightarrow} 
\cdots 
\overset{\varphi_{l-1}}{\longrightarrow} 
(X_l, \Delta_l)=:(X^*, \Delta^*)
$$ 
over $S$ 
such that $(X_i ,\Delta_i)$, where 
$\Delta_i:={\varphi_{i-1}}_*\Delta_{i-1}$, 
is generalized MR log canonical 
for every $i$ and that one of the followings 
holds. 
\begin{itemize}
\item[(i)] {\em{(Good minimal model).}}~If 
$K_X+\Delta$ is pseudo-effective 
over $S$, 
then $K_{X^*}+\Delta^*$ is semi-ample over $S$. 
\item[(ii)] {\em{(Mori fiber space).}}~If 
$K_X+\Delta$ is not pseudo-effective over $S$, 
then there is a contraction morphism 
$g\colon X^*\to W$ over $S$ onto a normal variety $W$, which is 
projective over $S$, with 
connected fibers such that 
$-(K_{X^*}+\Delta^*)$ is $g$-ample, $\dim W<2$, 
and the relative Picard number $\rho(X^*/W)$ is one.  
\end{itemize}
Moreover, we have the following properties. 
\begin{itemize}
\item[(1)] If $(X, \Delta)$ is log canonical, then 
so is $(X_i, \Delta_i)$ for every $i$. 
\item[(2)] If $X$ is $\mathbb Q$-factorial, 
then so is $X_i$ for every $i$. 
\item[(3)] If $(X, \Delta)$ is MR log canonical, 
then so is $(X_i,\Delta_i)$ for every $i$. 
\end{itemize}
\end{thm}

More precisely, we prove the cone and contraction theorem 
for log surfaces 
in full generality (see Theorem \ref{a-thm5.5}). 
Let $\varphi\colon X\to Z$ be a $(K_X+\Delta)$-negative 
extremal birational contraction. 
Then we can check that if $(X, \Delta)$ is generalized 
MR log canonical then so is $(Z,  \varphi_*\Delta)$ 
(see Lemma \ref{a-lem4.3}). 
Therefore, we can freely run the minimal model program for 
generalized MR log canonical surfaces. 
Hence we establish the minimal model theory for 
generalized MR log canonical surfaces. 
We note that Theorem \ref{a-thm1.5} contains 
the abundance theorem for generalized MR log canonical 
surfaces, which is highly nontrivial and is new. 

\begin{ack}
The author was partially supported by JSPS KAKENHI 
Grant Numbers JP16H03925, JP16H06337. 
He thanks Kento Fujita, 
Kenta Hashizume, Haidong Liu, and Hiromu Tanaka for useful comments. 
He also thanks Keisuke Miyamoto for pointing out some mistakes. 
Finally, he thanks the referee for some useful comments and 
suggestions. 
\end{ack}

We will treat algebraic surfaces defined over any algebraically 
closed field $k$ of arbitrary characteristic throughout this paper. 
We will freely use the standard notation and the results in 
\cite{fujino-fundamental},  
\cite{fujino-surfaces}, and \cite{tanaka} (see also \cite{fujino-tanaka} and 
\cite{fujino-foundations}). 

\section{Preliminaries}\label{a-sec2}

In this section, we collect some basic definitions. 

\begin{defn}[Operations for $\mathbb R$-divisors]\label{a-def2.1}
Let $D=\sum _i d_i D_i$ be an $\mathbb R$-divisor on a normal surface 
such that 
$D_i$ is a prime 
divisor and $d_i$ is a real number for every $i$ and 
that $D_i\ne D_j$ for $i\ne j$. 
Then $\lceil D\rceil$ (resp.~$\lfloor D\rfloor$) 
denotes the {\em{round-up}} (resp.~{\em{round-down}}) 
of $D$. 
We put $\{D\}:=D-\lfloor D\rfloor$ and call it 
the {\em{fractional part}} of $D$. 
\end{defn}

\begin{defn}[$\mathbb R$-linear equivalence]\label{a-def2.2}
Let $\Delta_1$ and $\Delta_2$ be $\mathbb R$-Cartier divisors 
on a normal surface. 
Then $\Delta_1\sim _{\mathbb R}\Delta_2$ menas 
that $\Delta_1$ is $\mathbb R$-linearly equivalent to 
$\Delta_2$. 
\end{defn}

We recall the definitions of {\em{klt surfaces}}, 
{\em{log canonical surfaces}}, and {\em{multiplier ideal sheaves}}. 

\begin{defn}[KLT surfaces, log canonical surfaces, and multiplier 
ideal sheaves]\label{a-def2.3}
Let $(X, \Delta)$ be a log surface as in Definition \ref{a-def1.1} and 
let $f\colon Y\to X$ be a proper 
birational morphism 
from a normal surface $Y$. 
Then we can write $$
K_Y+\Delta_Y=f^*(K_X+\Delta)
$$ with $f_*\Delta_Y=\Delta$.  
If the coefficients of $\Delta_Y$ are less than or equal to 
one (resp.~less than one) 
for every $f\colon Y\to X$, then $(X, \Delta)$ is 
called a {\em{log canonical surface}} 
(resp.~{\em{klt surface}}). 

We further assume that $Y$ is smooth and $\Supp \Delta_Y$ is 
a simple normal crossing divisor on $Y$. 
Then we put 
$$
\mathcal J(X, \Delta):=f_*\mathcal O_Y(-\lfloor \Delta_Y\rfloor)
=f_*\mathcal O_Y(\lceil K_Y-f^*(K_X+\Delta)\rceil)
$$
and call it the {\em{multiplier ideal sheaf}} of the pair $(X, \Delta)$. 
It is well known that $\mathcal J(X, \Delta)$ is 
independent of the resolution $f\colon Y\to X$. 
We note that $\mathcal J(X, \Delta)\subset \mathcal O_X$ 
holds, that is, $\mathcal J(X, \Delta)$ is an ideal sheaf on $X$. 
We can check that $(X, \Delta)$ is klt if and only if 
$\mathcal J(X, \Delta)=\mathcal O_X$ holds. 
\end{defn}

Let us recall a useful vanishing theorem 
for log surfaces. 

\begin{thm}[Kawamata--Viehweg--Nadel vanishing theorem 
for log surfaces]\label{a-thm2.4} 
Let $(X, \Delta)$ be a log surface and 
let $\pi\colon X\to S$ be a proper surjective morphism onto a variety 
$S$. 
When the characteristic of the base field $k$ is positive, we further 
assume that $\dim S\geq 1$. 
Let $L$ be a Cartier divisor on $X$ such that 
$L-(K_X+\Delta)$ is $\pi$-nef and $\pi$-big. 
Then $$R^i\pi_*(\mathcal O_X(L)\otimes \mathcal J(X, \Delta))=0$$ 
holds for 
every $i>0$, 
where $\mathcal J(X, \Delta)$ 
is the multiplier ideal sheaf of $(X, \Delta)$. 
\end{thm}

\begin{proof}[Sketch of Proof]
If the characteristic of $k$ is zero, 
then we can assume that 
$k=\mathbb C$ by the Lefschetz principle. 
In this case, the statement is well known 
(see \cite[Theorem 3.4.2]{fujino-foundations}). 
If the characteristic of $k$ is positive, 
then the desired vanishing theorem follows from 
\cite[Theorem 2.7]{tanaka2} (see also \cite[Section 3]{tanaka3}). 
\end{proof}

We close this section with the following well-known result on projectivity. 

\begin{lem}\label{a-lem2.5}
Let $X$ be a normal $\mathbb Q$-factorial surface. 
Then $X$ is always quasi-projective. 
\end{lem}

\begin{proof} 
The proof of \cite[Lemma 2.2]{fujino-surfaces} works as well 
over 
any algebraically closed field $k$ of arbitrary characteristic. 
\end{proof}

\section{MR log canonical surfaces and generalized MR log canonical 
surfaces}\label{a-sec3}

In this section, we explain various definitions and some examples. 
The following lemmas are obvious by definition. 
We state them explicitly for the reader's convenience. 

\begin{lem}\label{a-lem3.1} 
Let $(X, \Delta)$ be a log canonical surface. 
Then $(X, \Delta)$ is MR log canonical. 
\end{lem}
\begin{proof}
Let $f\colon Y\to X$ be the minimal resolution of 
singularities of $X$. Then we can write $$K_Y+\Delta_Y=f^*(K_X+\Delta)$$ 
as usual. By the definition of log canonical 
surfaces (see 
Definition \ref{a-def2.3}), 
the coefficients of $\Delta_Y$ are less than or equal to one. 
This means that $(X, \Delta)$ is MR log canonical. 
\end{proof}

\begin{lem}\label{a-lem3.2}
Let $(X, \Delta)$ be a log surface such that 
$X$ is smooth. 
Then $(X, \Delta)$ is MR log canonical. 
\end{lem}

\begin{proof}
Since $X$ is smooth, $X$ itself is the minimal resolution of 
singularities of $X$. 
Therefore, $(X, \Delta)$ is MR log canonical because 
the coefficients of $\Delta$ are less than or equal to one 
by the definition of log surfaces (see Definition \ref{a-def1.1}). 
\end{proof}

We note that an MR log canonical surface is not necessarily 
log canonical. 

\begin{ex}\label{a-ex3.3}
We put $X=\mathbb P^2$ and 
$\Delta=L_1+L_2+L_3$, 
where 
$L_1$, $L_2$, and $L_3$ are three distinct 
lines passing through a point $P$. 
Then $(X, \Delta)$ is MR log canonical by Lemma \ref{a-lem3.2}. 
On the other hand, we can check that $(X, \Delta)$ is 
not log canonical by taking a blow-up of $X$ at $P$. 
\end{ex}

We summarize the relationship between various definitions. 

\begin{lem}\label{a-lem3.4}
We have the following properties. 
\begin{itemize}
\item[(i)] If $(X, \Delta)$ is MR log canonical, 
then it is automatically generalized MR log canonical. 
\item[(ii)] If $(X, \Delta)$ is a $\mathbb Q$-factorial 
log surface, 
then it is generalized MR log canonical. 
\item[(iii)] A log canonical surface is not necessarily $\mathbb Q$-factorial. 
\item[(iv)] A $\mathbb Q$-factorial log surface is not 
always 
log canonical. 
\item[(v)] A $\mathbb Q$-factorial 
log surface is not always MR log canonical. 
\end{itemize}
\end{lem}

\begin{proof} By definition, (i) is obvious because 
smooth surface is automatically $\mathbb Q$-factorial. 
Let $(X, \Delta)$ be a $\mathbb Q$-factorial 
log surface. Then the identity map of $X$ satisfies the condition of 
Definition \ref{a-def1.4}. Therefore, 
$(X, \Delta)$ is generalized MR log canonical. 
This is (ii). It is well known that a log canonical surface is 
not always $\mathbb Q$-factorial by the classification of 
log canonical surface singularities. The simplest example is 
the cone over an elliptic curve. Hence (iii) holds. 
We note that surfaces with only rational singularities are 
automatically $\mathbb Q$-factorial. On the other hand, 
two-dimensional rational singularities are not necessarily 
log canonical. Therefore, we obtain (iv). 
We can easily construct a projective toric 
surface $X$ with only one singular point $P$ of 
type $A_1$. It is well known 
that $X$ is $\mathbb Q$-factorial. 
Let $D_1$ and $D_2$ be two general Cartier divisors 
on $X$ passing through $P$. 
Then $(X, \Delta)$, where $\Delta=D_1+D_2$, 
is a $\mathbb Q$-factorial log surface. 
However, $(X, \Delta)$ is not MR log canonical at $P$. Hence 
we have (v). Of course, (iv) is a special case of (v). 
\end{proof}

By (ii) and (v) in 
Lemma \ref{a-lem3.4}, a generalized MR log canonical surface 
is not always MR log canonical. 
Moreover, there exists a generalized MR log canonical surface 
which is not $\mathbb Q$-factorial by Lemma \ref{a-lem3.1} and Lemma 
\ref{a-lem3.4} (iii). 

\section{Basic properties of generalized MR log canonical 
surfaces}\label{a-sec4} 

Let us start with 
the abundance theorem for GMRLC surfaces. 
It easily follows from 
\cite[Theorem 8.1]{fujino-surfaces} and 
\cite[Theorem 6.7]{tanaka}. 

\begin{thm}[Abundance theorem for GMRLC surfaces]\label{a-thm4.1}
Let $(X, \Delta)$ be a GMRLC surface and 
let $\pi\colon X\to S$ be a proper surjective morphism onto 
a variety $S$.  
Assume that $K_X+\Delta$ is $\pi$-nef. 
Then $K_X+\Delta$ is $\pi$-semi-ample. 
\end{thm}
\begin{proof}
By definition, there exist a proper 
birational morphism $f\colon Y\to X$ such that 
$K_Y+\Delta_Y=f^*(K_X+\Delta)$ and that $(Y, \Delta_Y)$ is a 
$\mathbb Q$-factorial log surface.  
It is obvious that 
$K_Y+\Delta_Y$ is $\pi\circ f$-nef. 
Therefore, $K_Y+\Delta_Y$ is $\pi\circ f$-semi-ample by 
the abundance theorem for $\mathbb Q$-factorial 
log surfaces (see \cite[Theorem 8.1]{fujino-surfaces} 
and \cite[Theorem 6.7]{tanaka}). 
Thus, we obtain that $K_X+\Delta$ is $\pi$-semi-ample. 
\end{proof}

As an easy consequence of Theorem \ref{a-thm4.1}, 
we have a weak version of the basepoint-free theorem 
for GMRLC surfaces. This weak version seems to be 
almost sufficient for 
many applications. 
We will treat more general results in Theorems \ref{a-thm5.1} and 
\ref{a-thm5.4}. 

\begin{cor}\label{a-cor4.2} 
Let $(X, \Delta)$ be a GMRLC surface and 
let $\pi\colon X\to S$ be a projective 
morphism onto a variety $S$. 
Let $D$ be a $\pi$-nef Cartier divisor 
on $X$ such that $D-(K_X+\Delta)$ is $\pi$-ample. 
Then $D$ is $\pi$-semi-ample. 
\end{cor}
\begin{proof}
Without loss of generality, we may assume that $S$ is 
affine and that $\pi_*\mathcal O_X\simeq \mathcal O_S$ holds. 
By Bertini's theorem, we can write 
$D\sim _{\mathbb R} K_X+\Delta+A+\pi^*B$ for some 
ample $\mathbb R$-divisor $A$ on $X$ and 
an $\mathbb R$-Cartier divisor $B$ on $S$ such that 
$(X, \Delta+A)$ is GMRLC. 
By Theorem \ref{a-thm4.1}, $K_X+\Delta+A$ is $\pi$-semi-ample since 
it is $\pi$-nef. 
Hence $D$ is $\pi$-semi-ample. 
\end{proof}

The following lemma is a key result for GMRLC surfaces. 

\begin{lem}\label{a-lem4.3}
Let $(X, \Delta)$ be a GMRLC {\em{(}}resp.~an MRLC{\em{)}} 
surface and let $\varphi\colon X\to Z$ be a proper birational morphism 
onto a normal surface $Z$ such that 
$-(K_X+\Delta)$ is $\varphi$-nef. 
Then $(Z, \Delta_Z)$ is GMRLC {\em{(}}resp.~MRLC{\em{)}}, 
where $\Delta_Z=\varphi_*\Delta$. 
\end{lem}
\begin{proof}
First, we assume that $(X, \Delta)$ is GMRLC. 
We take a proper birational morphism 
$f\colon Y\to X$ 
from a $\mathbb Q$-factorial normal surface $Y$ 
with $K_Y+\Delta_Y=f^*(K_X+\Delta)$ 
as in Definition \ref{a-def1.4}. 
Since $-(K_X+\Delta)$ is $\varphi$-nef, 
$-(K_Y+\Delta_Y)$ is nef over $Z$. 
Then, by the negativity lemma, 
we can uniquely take an $\mathbb R$-divisor 
$\Theta$ on $Y$ with $\Theta\leq \Delta_Y$ such that 
$-(K_Y+\Theta)$ is numerically trivial over $Z$ and 
$(\varphi\circ f)_*\Theta=\Delta_Z$. 
We write $\Theta=\Theta_+-\Theta_-$, where 
$\Theta_+$ and $\Theta_-$ are both effective $\mathbb R$-divisors 
with no common components. 
By construction, 
$K_Y+\Theta_+$ is numerically equivalent to 
$\Theta_-$ over $Z$ and $\Theta_-$ is exceptional over $Z$. 
We note that $Y$ is projective over $Z$ since it is $\mathbb Q$-factorial 
(see Lemma \ref{a-lem2.5}). 
Thus we can run the minimal model program 
over $Z$ with respect to $K_Y+\Theta_+$ 
(see \cite[Theorem 3.3]{fujino-surfaces} and \cite[Theorem 6.5]{tanaka}). 
Then we finally get a $\mathbb Q$-factorial log surface $(Y^*, 
\Theta_+^*)$ such that 
$K_{Y^*}+\Theta_+^*$ is numerically trivial over $Z$. 
By the abundance theorem (see \cite[Theorem 8.1]{fujino-surfaces} 
and \cite[Theorem 6.7]{tanaka}), 
$K_{Y^*}+\Theta_+^*$ is $\mathbb R$-linearly equivalent to zero 
over $Z$. 
By construction, $p_*(K_{Y^*}+\Theta_+^*)=K_Z+\Delta_Z$ holds, 
where $p\colon Y^*\to Z$, 
and $K_Z+\Delta_Z$ is $\mathbb R$-Cartier. 
Thus, we have $K_{Y^*}+\Theta_+^*=p^*(K_Z+\Delta_Z)$. 
This means that $(Z, \Delta_Z)$ is GMRLC because 
$(Y^*, \Theta_+^*)$ is a $\mathbb Q$-factorial 
log surface. 

Next, we assume that $(X, \Delta)$ is MRLC. 
Then we may assume that $f\colon Y\to X$ is the minimal resolution 
of singularities of $X$ in the above argument. 
Then we can easily see that $Y^*$ is a smooth 
surface. 
More precisely, $Y\to Y^*$ is a composition of 
contractions of $(-1)$-curves. 
If there exists a $(-1)$-curve in the fiber of $p\colon Y^*\to Z$, 
then we contract it to a smooth point. 
By repeating this process finitely many times, 
we may further assume that 
$p\colon Y^*\to Z$ is the minimal resolution of singularities 
of $Z$. 
Therefore, $(Z, \Delta_Z)$ is MRLC. 
\end{proof}

As an easy application of Lemma \ref{a-lem4.3}, 
we prove the following useful lemma although 
we do not use it explicitly in this paper. 

\begin{lem}\label{a-lem4.4} 
Let $(X, \Delta)$ be a GMRLC {\em{(}}resp.~an 
MRLC{\em{)}} surface and 
let $\Delta'$ be an effective $\mathbb R$-divisor on $X$ 
such that $\Delta'\leq \Delta$. 
Then $(X, \Delta')$ is GMRLC {\em{(}}resp.~MRLC{\em{)}}. 
\end{lem}
\begin{proof}
We take a proper birational morphism $f\colon Y\to X$ from a 
$\mathbb Q$-factorial normal surface $Y$ with 
$K_Y+\Delta_Y=f^*(K_X+\Delta)$ as in Definition \ref{a-def1.4}. 
We note that $f^{-1}_*(\Delta-\Delta')$ is $f$-nef. 
Therefore, we obtain that 
$-\left(K_Y+\Delta_Y-f^{-1}_*(\Delta-\Delta')\right)$ is $f$-nef. 
We also note that 
$f_*\left(\Delta_Y-f^{-1}_*(\Delta-\Delta')\right)=\Delta'$ by construction. 
Of course, $\left(Y, \Delta_Y-f^{-1}_*(\Delta-\Delta')\right)$ is 
GMRLC since $Y$ is $\mathbb Q$-factorial. 
Therefore, by Lemma \ref{a-lem4.3}, 
$(X, \Delta')$ is GMRLC. 
When $(X, \Delta)$ is MRLC, we can 
assume that $Y$ is smooth in the above argument. 
In this case, $\left(Y, \Delta_Y-f^{-1}_*(\Delta-\Delta')\right)$ is MRLC. 
Therefore, $(X, \Delta)$ is MRLC by Lemma \ref{a-lem4.3}. 
\end{proof}

\section{Minimal model theory for generalized MR log canonical surfaces}
\label{a-sec5}

In this section, we discuss the minimal model theory for 
GMRLC surfaces. 
Let us start with the basepoint-free theorem 
for log surfaces (see also \cite[Theorem 4.2]{tanaka3}). 

\begin{thm}[Basepoint-freeness for log surfaces I]\label{a-thm5.1} 
Let $(X, \Delta)$ be a log surface such that 
$\lfloor \Delta\rfloor=0$ and let $\pi\colon X\to S$ be a projective 
surjective morphism onto a variety $S$. 
Let $D$ be a $\pi$-nef Cartier divisor 
on $X$ such that $aD-(K_X+\Delta)$ is 
$\pi$-nef and $\pi$-big for some 
positive integer $a$. 
Then there exists a positive integer $m_0$ 
such that 
$\mathcal O_X(mD)$ is $\pi$-generated for 
every integer $m\geq m_0$. 
\end{thm}

Before we explain the proof of Theorem \ref{a-thm5.1}, 
we prepare two easy non-vanishing theorems. 

\begin{lem}[Non-vanishing theorem I]\label{a-lem5.2}
Let $C$ be a smooth projective 
curve and let $G$ be a $\mathbb Q$-divisor 
on $C$ such that 
$\lfloor G\rfloor \leq 0$. 
Let $D$ be a nef Cartier divisor on $C$ such that 
$aD-(K_C+G)$ is ample for some positive 
integer $a$. 
Then $H^0(C, \mathcal O_C(mD+\lceil -G\rceil))\ne 0$ for 
every integer $m\geq a$. 
\end{lem}

\begin{proof}
If $C\simeq \mathbb P^1$, then $H^0(C, \mathcal O_C(mD+\lceil -G
\rceil))\ne 0$ for every non-negative 
integer $m$ since $\deg D\geq 0$ and $\lceil 
-G\rceil \geq 0$. If $C\not \simeq \mathbb P^1$, then we can easily check 
that $H^0(C, \mathcal O_C(mD+\lceil -G
\rceil))\ne 0$ for every $m\geq a$ by the Riemann--Roch 
formula. 
More precisely, we have 
\begin{equation*}
\begin{split}
&\dim H^0(C, \mathcal O_C(mD+\lceil -G
\rceil))-\dim H^1(C, \mathcal O_C(mD+\lceil -G
\rceil))
\\&=\deg (mD+\lceil -G\rceil)-g+1>
\deg (K_C+\{G\})-g+1\\&\geq 2g-2-g+1=g-1\geq 0 
\end{split}
\end{equation*}
for every $m\geq a$, where $g$ denotes the genus of 
$C$. 
\end{proof}

\begin{lem}[Non-vanishing theorem II]\label{a-lem5.3}
Let $(X, \Delta)$ be a klt surface and let $\pi\colon X\to S$ be 
a proper surjective morphism 
with $\dim S\geq 1$. 
Let $D$ be a $\pi$-nef Cartier divisor on $X$ such that 
$aD-(K_X+\Delta)$ is $\pi$-nef 
and $\pi$-big for some 
positive integer $a$. 
Then there exists a positive integer $m'$ such that 
$\pi_*\mathcal O_X(mD)\ne 0$ for every integer $m\geq m'$. 
\end{lem}

\begin{proof}
If $\dim S=2$, then $\pi_*\mathcal O_X(mD)\ne 0$ obviously 
holds for every integer $m$. 
Therefore, we may assume that 
$S$ is a smooth curve with $\pi_*\mathcal O_X\simeq 
\mathcal O_S$. 
We may further assume that 
every fiber of $\pi\colon X\to S$ is irreducible by shrinking $S$ 
suitably. 
In this case, $D$ is $\pi$-ample or 
$\pi$-numerically trivial since $\pi$ is flat. 
When $D$ is $\pi$-ample, the statement 
is obvious. 
When $D$ is $\pi$-numerically trivial, $mD-(K_X+\Delta)$ is $\pi$-nef 
and $\pi$-big for every integer $m$. 
Therefore, $R^i\pi_*\mathcal O_X(mD)=0$ for 
every integer $m$ and every $i>0$ by Theorem \ref{a-thm2.4}. 
This implies 
that $\pi_*\mathcal O_X(mD)\ne 0$ for every integer 
$m$ since $\pi$ is flat and $\pi_*\mathcal O_X\ne 0$. 
Note that 
$$
\dim H^0(F, \mathcal O_F(mD))=\chi (F, \mathcal O_F(mD))
=\chi (F, \mathcal O_F)=\dim H^0(F, \mathcal O_F)>0
$$ 
holds for every integer $m$, where $F$ is the 
generic fiber of $\pi\colon X\to S$. 
\end{proof}

Let us explain the proof of Theorem \ref{a-thm5.1}. 
We only give a sketch of the proof because 
the arguments are more or less well known. 

\begin{proof}[Sketch of Proof of Theorem \ref{a-thm5.1}]
Here, we will explain how to prove Theorem \ref{a-thm5.1}. 
Without loss of generality, we may assume that 
$S$ is affine and $\pi_*\mathcal O_X\simeq \mathcal O_S$. 
By Kodaira's lemma (see 
\cite[Lemma 2.1.29]{fujino-foundations}) and perturbing 
the coefficients of $\Delta$ slightly, 
we further assume that 
$aD-(K_X+\Delta)$ is $\pi$-ample and 
that $\Delta$ is a $\mathbb Q$-divisor. 
\begin{case}
We assume that the characteristic of the base field $k$ is positive and 
$\dim S=0$. 

If $D$ is not numerically trivial, then the statement follows 
from \cite[Theorem 3.2]{tanaka2}. 
If $D$ is numerically trivial, then $\Pic (X)$ is 
a free abelian group of finite rank by \cite[Theorem 1.5]{ohta-okawa} 
since $-(K_X+\Delta)$ is ample. 
Therefore, we have $D\sim 0$. 
Hence we obtain the desired statement when $\dim S=0$ and 
the characteristic of $k$ is positive. 
\end{case}
\begin{case}
We assume that $(X, \Delta)$ is klt and 
the characteristic of the base field $k$ is zero. 

This case is a special case of 
the well-known basepoint-free theorem 
for klt pairs. 
It can be proved by the traditional X-method, that is, 
a clever application of the Kawamata--Viehweg vanishing 
theorem mainly due to Yujiro Kawamata. 
For the details of 
the X-method, see 
\cite[Section 4.2]{fujino-fundamental}. 
\end{case}
\begin{case}
We assume that $(X, \Delta)$ is not klt. 
We further assume that $\dim S\geq 1$ holds 
when the characteristic of 
the base field $k$ is positive. 

In this case, a slightly modified 
version of the X-method works 
by the vanishing theorem (see Theorem \ref{a-thm2.4}) 
and the non-vanishing theorem (see Lemma \ref{a-lem5.2}). 
There are no difficulties to adapt the arguments in the 
proof of \cite[Theorem 3.2]{tanaka2} for our setting. 
\end{case}
\begin{case}
We assume that the characteristic of 
the base field 
$k$ is positive, $\dim S\geq 1$, and $(X, \Delta)$ is klt. 

In this case, we first use 
Lemma \ref{a-lem5.3} and apply the traditional X-method. 
Then we can obtain the desired basepoint-freeness 
by the vanishing theorem (see Theorem \ref{a-thm2.4}) and 
the non-vanishing theorem (see Lemma \ref{a-lem5.2}). 
\end{case}
Therefore, we obtain the desired basepoint-freeness 
for log surfaces. 
\end{proof}

As a corollary of Theorem \ref{a-thm5.1}, we have 
the following statement. 
Theorem \ref{a-thm5.4} is much sharper than 
Corollary \ref{a-cor4.2} because 
a GMRLC surface is a log surface by definition. 

\begin{thm}[Basepoint-freeness for log surfaces II]\label{a-thm5.4}
Let $(X, \Delta)$ be a log surface and 
let $\pi\colon X\to S$ be a projective morphism onto a variety $S$. 
Let $D$ be a $\pi$-nef Cartier divisor 
on $X$ such that $aD-(K_X+\Delta)$ is 
$\pi$-ample for some positive 
integer $a$. 
Then there exists a positive integer $m_0$ such that 
$\mathcal O_X(mD)$ is $\pi$-generated 
for every integer $m\geq m_0$. 
\end{thm}

\begin{proof}
Without loss of generality, we may assume that 
$S$ is affine and $\pi_*\mathcal O_X\simeq \mathcal O_S$. 
We take a $\pi$-ample 
Cartier divisor $H$ on $X$. 
Then we can take a large positive integer $b$ and 
an effective Weil divisor $G$ such that 
$\lfloor \Delta\rfloor +bH\sim G\ge 0$ and that 
$\Supp G$ contains no irreducible components of $\lfloor 
\Delta\rfloor$. 
By replacing $\Delta$ with $\Delta-\varepsilon \lfloor \Delta\rfloor 
+\varepsilon G$, 
which is $\mathbb Q$-linearly 
equivalent to $\Delta+\varepsilon bH$, for $0<\varepsilon \ll 1$, we may 
further assume that 
$\lfloor \Delta\rfloor=0$. Then this theorem follows 
from Theorem \ref{a-thm5.1}. 
\end{proof}

For the reader's convenience, we explicitly 
state the cone and contraction theorem for log surfaces 
(see also \cite[Theorem 4.4]{tanaka3}). 

\begin{thm}[Cone and contraction theorem for 
log surfaces]\label{a-thm5.5} 
Let $(X, \Delta)$ be a log surface and let $\pi\colon X\to S$ 
be a projective morphism onto a variety $S$. 
Then the equality 
$$
\NE(X/S)=\NE(X/S)_{K_X+\Delta\geq 0}
+\sum _j R_j
$$ 
and the following hold. 
\begin{itemize}
\item[(i)] $R_j$ is a $(K_X+\Delta)$-negative 
extremal ray of $\NE(X/S)$ for every $j$. 
\item[(ii)] Let $H$ be a $\pi$-ample 
$\mathbb R$-divisor on $X$. 
Then there are only finitely many $R_j$'s 
included in $(K_X+\Delta+H)_{<0}$. 
In particular, the $R_j$'s are discrete 
in the half-space $(K_X+\Delta)_{<0}$. 
\item[(iii)] Let $R$ be a $(K_X+\Delta)$-negative 
extremal ray of $\NE(X/S)$. 
Then there exists a contraction morphism 
$\varphi_R\colon X\to Y$ over $S$ with 
the following properties.  
\begin{itemize}
\item[(a)] Let $C$ be an irreducible curve on $X$ such that 
$\pi(C)$ is a point. Then $\varphi_R(C)$ is a point if and only if $[C]\in R$. 
\item[(b)] $\mathcal O_Y\simeq 
(\varphi_R)_*\mathcal O_X$. 
\item[(c)] Let $\mathcal L$ be a line bundle on $X$ such that 
$\mathcal L\cdot C=0$ for 
every curve $C$ with $[C]\in R$. 
Then there exists a line bundle $\mathcal M$ on $Y$ 
such that $\mathcal L\simeq 
\varphi^*_R\mathcal M$ holds. 
\end{itemize}
\end{itemize}
\end{thm}

We note that the cone and contraction theorem 
holds for log surfaces $(X, \Delta)$ without any extra assumptions 
on $(X, \Delta)$. 

\begin{proof}[Sketch of Proof of Theorem \ref{a-thm5.5}]
If the characteristic of the base field $k$ is zero, 
then Theorem \ref{a-thm5.5} follows 
from \cite[Theorem 3.2]{fujino-surfaces}, 
which is a special case of \cite[Theorem 1.1]{fujino-fundamental}. 
Although \cite[Theorem 3.2]{fujino-surfaces} (see 
also \cite[Theorem 1.1]{fujino-fundamental}) is 
stated only for the case where $k=\mathbb C$, it holds 
true under the assumption that 
the characteristic of $k$ is zero. 
This is because the Kodaira-type vanishing theorems 
in \cite{fujino-fundamental} hold true when the 
characteristic of $k$ is zero by the Lefschetz principle. 
For the details of the cone and contraction theorem 
in characteristic zero, we strongly recommend the 
reader to see \cite{fujino-fundamental}. 

From now on, we assume that 
the characteristic of $k$ is positive. 
Then the cone theorem, that is, (i) and (ii), holds true 
by \cite[Theorem 6.2]{tanaka}. 
Therefore, it is sufficient to prove the contraction theorem (iii). 
We can take a $\pi$-nef Cartier divisor $D$ on $X$ such that 
$R=\NE(X/S)\cap D^{\perp}$ by (ii). 
Then $aD-(K_X+\Delta)$ is $\pi$-ample for some 
positive integer $a$. 
By Theorem \ref{a-thm5.4}, there exists a positive 
integer $m_0$ such that 
$\mathcal O_X(mD)$ is $\pi$-generated for every 
integer $m\geq m_0$. 
We take the Stein factorization of the associated 
morphism. 
Then we have a contraction morphism 
$\varphi_R\colon X\to Y$ over $S$ satisfying (a) and (b). 
By construction, $-(K_X+\Delta)$ is $\varphi_R$-ample. 
Therefore, $\mathcal L-(K_X+\Delta)$ is also 
$\varphi_R$-ample since $\mathcal L$ is relatively numerically 
trivial over $Y$. 
By Theorem \ref{a-thm5.4} again, 
we see that there exists 
a positive integer $n_0$ such that 
$\mathcal L^{\otimes n}$ is $\varphi_R$-generated 
for every integer $n\geq n_0$. 
This implies (c). 
\end{proof}

We note that $R_j$ in Theorem \ref{a-thm5.5} 
is spanned by a rational curve. 

\begin{thm}[{Extremal rational curves, 
see \cite[Proposition 3.8]{fujino-surfaces}}]
\label{a-thm5.6}
Let $(X, \Delta)$ be a log surface and let $\pi\colon X\to S$ be a projective 
morphism onto a variety $S$. 
Let $R$ be a $(K_X+\Delta)$-negative extremal ray of $\NE(X/S)$. 
Then $R$ is spanned by a rational curve $C$ on $X$ such that 
$$0<-(K_X+\Delta)\cdot C\leq 3. 
$$ 
Moreover, if $X\not\simeq \mathbb P^2$, then we can 
choose $C$ with $$0<-(K_X+\Delta)\cdot C\leq 2. 
$$ 
\end{thm}

\begin{proof}
By the vanishing theorem in \cite[Theorem 6.2]{fujino-tanaka}, 
which holds true over any algebraically closed field $k$ 
of arbitrary characteristic, 
the proof of \cite[Proposition 3.8]{fujino-surfaces} 
works without any changes. 
\end{proof}

Finally, we give a sketch of the proof of Theorem \ref{a-thm1.5}, which 
is the main result of this short paper. 

\begin{proof}[Sketch of Proof of Theorem \ref{a-thm1.5}]
Since $(X, \Delta)$ is a log surface that is projective 
over $S$, 
we can apply the cone and contraction theorem 
to $(X, \Delta)$ (see Theorem \ref{a-thm5.5}). 
By Lemma \ref{a-lem4.3}, 
we can run the minimal model program under the assumption that 
$(X, \Delta)$ is GMRLC. 
After finitely many steps, the minimal model program terminates. 
By construction, $K_X+\Delta$ is pseudo-effective 
over $S$ if and only if so is $K_{X^*}+\Delta^*$. 
By the abundance theorem (see Theorem \ref{a-thm4.1}), 
$K_{X^*}+\Delta^*$ is semi-ample over $S$ when 
$K_{X^*}+\Delta^*$ is nef over $S$. 

If $(X, \Delta)$ is log canonical, then we can easily see that 
$(X_i, \Delta_i)$ is also log canonical 
for every $i$ by the negativity lemma. 
Hence (1) is the classical minimal model theory for log canonical 
surfaces. For the details, see \cite{fujino-surfaces} and 
\cite{tanaka} (see also \cite{fujino-tanaka}). 

The case where $(X, \Delta)$ is a $\mathbb Q$-factorial log surface 
(see (2)) was established in characteristic $0$ and 
$p>0$ in \cite{fujino-surfaces} and \cite{tanaka}, 
respectively. 

If $(X, \Delta)$ is MRLC, then $(X_i, \Delta_i)$ is also 
MRLC by Lemma \ref{a-lem4.3}. 
Therefore, (3) holds. 
Note that (3) is the minimal model program sketched 
in \cite[Section 10]{alexeev}. 
\end{proof}

Theorem \ref{a-thm1.5} covers the minimal model 
theory for log canonical surfaces, $\mathbb Q$-factorial 
log surfaces, and 
MR log canonical surfaces. 

\medskip 

We note that some results in this short paper were already discussed 
in a more general setting by Hiromu Tanaka. 
We recommend the interested reader to see \cite{tanaka3}. 
We also note that the minimal model theory established in \cite{fujino-surfaces} 
was generalized for log surfaces in Fujiki's class $\mathcal C$ in 
\cite{fujino-fujiki}. 

\end{document}